\documentclass{article} 

\usepackage{amsmath,amssymb,amsthm,amscd}
\usepackage{graphics, graphpap}

\newtheorem{theorem}{Theorem}[section]
\newtheorem{corollary}[theorem]{Corollary}
\newtheorem{lemma}[theorem]{Lemma}

\theoremstyle{definition}
\newtheorem{definition}[theorem]{Definition}
\newtheorem{remark}[theorem]{Remark}

\def\n{\underline n}

\def\p{\frak p}
\def\m{\frak m}

\def\n{\mathbb{N}}

\def\Spec{\operatorname{Spec}}

\def\lim{\operatorname{\lim}}
\def\dlim{{\varinjlim}_n}
 
\def\dim{\operatorname{dim}}
\def\Ext{\operatorname{Ext}}

\def\N-dim{\operatorname{N-dim}}

\def\depth{\operatorname{depth}}
\def\f-depth{\operatorname{f-depth}}
\def\gdepth{\operatorname{gdepth}}
\def\Supp{\operatorname{Supp}}

\def\Ass{\operatorname{Ass}}

\def\ann{\operatorname{ann}}

\def\inf{\operatorname{inf}}

\newcommand{\fR}{\ensuremath{\mathcal R}}
\newcommand{\fN}{\ensuremath{\mathcal N}}

\begin{document}

\title{\normalsize ON THE FINITENESS AND STABILITY OF CERTAIN SETS OF ASSOCIATED PRIME IDEALS OF LOCAL COHOMOLOGY MODULES}
\author{Nguyen Tu Cuong$^{a}$ and Nguyen Van Hoang$^{b,c}$}
\date{\small $^{a}$Institute of Mathematics, Hanoi, Vietnam\\
$^{b}$Thai Nguyen University of Education, Thai Nguyen, Vietnam\\
$^{c}$Meiji Institute for Advanced Study of Mathematical Sciences,
Meiji University, Kawasaki, Japan
}
\maketitle

\vskip .5cm {\begin{quote} {\bf Abstract}{\footnote{{\it{Key words and phrases: }} associated prime, local cohomology, generalized local cohomology, $N$-sequence in dimension $>k$. \hfill\break
{\it{2000 Subject  Classification:}}  13D45, 13D07,  13C15 \hfill\break {$^a$E-mail: ntcuong@math.ac.vn}
\hfill\break {$^b$E-mail: nguyenvanhoang1976@yahoo.com}}. Let $(R,\frak{m})$  be a Noetherian local ring, $I$ an ideal of $R$ and $N$ a finitely generated $R$-module. Let $k{\ge}-1$ be an integer and $ r=\depth_k(I,N)$ the length of a maximal $N$-sequence in dimension $>k$ in $I$ defined by M. Brodmann and L. T. Nhan ({Comm. Algebra,  {\bf 36} (2008), 1527-1536). For a subset $S\subseteq \Spec R$ we set  $S_{{\ge}k}=\{\p\in S\mid\dim(R/\p){\ge}k\}$.  We first prove in  this paper  that  $\Ass_R(H^j_I(N))_{\ge k}$ is a finite set for all $j{\le}r$ }. Let  $\fN=\oplus_{n\ge 0}N_n$ be a finitely generated graded $\fR$-module, where  $\fR$ is a finitely generated standard graded algebra over $R_0=R$. Let $r$ be  the eventual value of $\depth_k(I,N_n)$. Then our second result says that for all $l{\le}r$ the sets $\bigcup_{j{\le}l}\Ass_R(H^j_I(N_n))_{{\ge}k}$ are stable for large $n$.}
\end{quote}

\section{Introduction} 

Let $(R,\m)$ be a Noetherian local ring, $I$ an ideal of $R$ and $N$ a finitely generated $R$-module. 
In 1990, C. Huneke \cite[Problem 4]{Hun} asked whether the set of associated primes of $H^j_I(N)$ is finite for all finitely generated modules $N$ and all $I$. Affirmative answers were given by Huneke-R.Y. Sharp \cite{HuS} and G. Lyubeznik  \cite{Lyu} for equicharacteristic regular local rings. Although, A. Singh  \cite{AS} and M. Katzman  \cite{Kat} provided  examples of finitely generated modules having some local cohomology modules with infinite associated prime ideals, the problem is still true in many situations, such as \cite{BF}, \cite{KhS1}, \cite{Lyu}, \cite{Mar}, \cite{Nh}. However, little is known about the finiteness of $\Ass_R(H^j_I(N)).$  Brodmann-L.T. Nhan introduced the notion of $N$-sequence in dimension $> k$ in \cite{BN}:  Let $k$ be an integer with $k\ge -1$. A sequence $x_1,\ldots,x_r$ of elements of $\m$ is called an $N$-sequence in dimension $> k$ if $x_i\notin\p$ for all $\p\in\Ass_R(N/(x_1,\ldots,x_{i-1})N)$ with $\dim(R/\p)> k$ and all $i=1,\ldots, r$. They also showed that every maximal $N$-sequence in dimension $>k$ in $I$ has the same length. This common length is denoted by $\depth_k(I,N)$. Note that $\depth_{-1}(I,N)$ is just $\depth(I,N)$, $\depth_0(I,N)$ is the filter depth $\f-depth(I,N)$ defined by R. L${\rm\ddot u}$-Z. Tang \cite{LT}, and $\depth_1(I,N)$ is the generalized depth $\gdepth(I,N)$ defined by Nhan \cite{Nh}.  Let $S$ be a subset of $\Spec(R)$ and $k\ge -1$ an integer. We set $S_{\ge k}=\{\p\in S\mid\dim(R/\p)\ge k\}$.   The first main result of this paper is to prove the following theorem.

\begin{theorem}\label{T1} Let $(R,\m)$ be a Noetherian local ring, $I$ an ideal of $R$ and $N$ a finitely generated $R$-module. Let $k$ be an integer with $k\ge -1$ and $r=\depth_k(I,N)$. If $r<\infty$ and $x_1,\ldots,x_r$ is an $N$-sequence in dimension $>k$ in $I$, then for any integer $j\le r$  the set $\Ass_R(H^j_I(N))_{\ge k}$ is finite. Moreover, we have 
$$\bigcup_{j\le l}\Ass_R(H^j_I(N))_{\ge k}=\bigcup_{j\le l}\Ass_R(N/(x_1,\ldots,x_j)N)_{\ge k}\cap V(I)$$
for all $l\le r$.
\end{theorem}

Let $\fR=\oplus_{n\ge 0}R_n$ be a finitely generated standard graded algebra over $R_0=R$ and $\oplus_{n\ge 0}N_n$ a finitely generated graded $\fR$-module.  In 1979, M. Brodmann \cite{Br1} had proved that the set $\Ass_R(N_n)$ is stable for large $n$. Based on this result he showed in \cite[Theorem 2 and Proposition 12]{Br2} that the integer $\depth(I,N_n)$ takes a constant value for large $n$. We  generalized this last one in \cite{CHK} and showed that   
$\depth_k(I,N_n)$ takes a constant value $r_k$ for large $n$. Moreover, in \cite{CHK}, we also prove that the set $\bigcup_{j\le r_1}\Ass_R(H^j_I(N_n))\cup\{\m\}$ is stable for large $n$, where $r_1$ is the stable value of $\depth_1(I,N_n)$ for large $n$.  Then the second main result of this paper is  the following theorem.

\begin{theorem}\label{T2} Let $(R,\m)$ be a Noetherian local ring and $I$ an ideal of $R$. Let $\fR=\oplus_{n\ge 0}R_n$ be a finitely generated standard graded algebra over $R_0=R$ and $\fN=\oplus_{n\ge 0}N_n$ a finitely generated graded $\fR$-module. For each integer $k\ge -1$, let $r$ be the eventual value of $\depth_k(I,N_n)$. Then for each integer $l\le r$ the set $\bigcup_{j\le l}\Ass_R(H^j_I(N_n))_{\ge k}$ is stable  for large $n$.
\end{theorem}

For two finitely generated $R$-modules $M$ and $N$,  J. Herzog \cite{Her} defined the generalized local cohomology module $H^j_I(M,N)$ of $M$ and $N$ with respect to $I$ by $$\displaystyle H^j_I(M,N)=\dlim\Ext^j_R(M/I^nM,N).$$  It is clear that 
if $M=R$, $H^j_I(R,N)$ is just the  ordinary local cohomology module $H^j_I(N)$ of $N$ with respect to $I$. Therefore all questions about local cohomology modules discussed above can be asked for generalized local cohomology modules. Thus, to prove the two theorems above, we will show in this paper the generalizations of them for generalized local cohomology.  Let us explain the organization  of the paper. 
 In the next  section, we present auxiliary results on generalized local cohomology modules, the  maximal length of an $N$-sequence in dimension $> k$ in certain ideals of $R$ and relationships between them. In Section 3, we show a generalization of Theorem \ref{T1} for generalized local cohomology modules (Theorem \ref{T11}). As consequences, we get again one of main results \cite[Theorem 3.1]{Nh}  of Nhan (see Corollary \ref{Csua1}) and  a generalization of the main results of Brodmann-Nhan \cite[Theorems 1.1 and 1.2]{BN} (see Corollary \ref{BN2}). The last section is devoted to prove Theorem \ref{T22}, which is a generalization of Theorem \ref{T2} for generalized local cohomology modules. Then we can derive from Theorem \ref{T22} a generalization of the main result of \cite[Theorem 1.2]{CHK} for generalized local cohomology modules (see Corollary \ref{C35}) and some new results about the asymptotic stability of the sets of associated prime ideals of Ext-modules (see Corollaries \ref{Ce} and \ref{BN3}).

\section{Generalized local cohomology module}
Throughout this paper  we use following notations: we denote by  $(R,\m)$  a Noetherian local ring with the maximal ideal $\m$,  and $M$, $N$ finitely generated $R$-modules; let $S$ be a subset of $ \Spec R$ and $k\ge -1$ an integer,  we set  $S_{\ge k}=\{\p\in S\mid\dim(R/\p)\ge k\}$ and  $S_{> k}=\{\p\in S\mid\dim(R/\p) > k\}$; for an ideal $I$ of $R$, $I_M$  is the annihilator of the $R$-module $M/IM$.

We first recall the notion of generalized local cohomology module defined by Herzog in \cite{Her}. 

\begin{definition} For an integer $j\ge 0$, the $j^{th}$ generalized local cohomology module $H^j_I(M,N)$ of $M$ and $N$ with respect to $I$  is defined by 
$$\displaystyle H^j_I(M,N)=\dlim\Ext^j_R(M/I^nM,N).$$ 
\end{definition}
It is clear that $H^j_I(R,N)$ is just the  ordinary local cohomology module $H^j_I(N)$ of $N$ with respect to $I$. On the other hand, Ext-modules can  be also determined  in many cases by generalized local cohomology modules.

\begin{lemma}\label{Lmbs} {\rm (cf. \cite[Lemma 2.1]{CH1})}  If $I\subseteq\ann(M)$ or $\Gamma_I(N)=N$ then $H^j_I(M,N)\cong\Ext^j_R(M,N)$ for all $j\ge 0$.
\end{lemma}
The following definition of a generalization of regular sequences is introduced by Brodmann-Nhan in \cite[Definition 2.1]{BN}.  

\begin{definition}\label{dn} Let $k\ge -1$ be an integer. A sequence $x_1,\ldots,x_r$ of elements of $\m$ is called an {\it $N$-sequence in dimension $> k$} if $x_i\notin\p$ for all $\p\in\Ass_R(N/(x_1,\ldots,x_{i-1})N)_{>k}$ and all $i=1,\ldots, r$.
\end{definition}

It is clear that $x_1,\ldots ,x_r$ is an $N$-sequence in dimension $>-1$ if and only if it is a regular sequence of $N$; and $x_1,\ldots ,x_r$ is an $N$-sequence in dimension $>0$ if and only if it is a filter regular sequence of $N$  introduced by P. Schenzel, N. V. Trung and the first author in \cite{Cst}. Moreover, $x_1,\ldots ,x_r$ is an $N$-sequence in dimension $>1$ if and only if it is a generalized regular sequence of $N$ defined by Nhan in \cite{Nh}.

\begin{remark}\label{RM1}  (i) It is easy to see that if $x_1,\ldots,x_r$ is an $N$-sequence in dimension $>k$, then so is $x_1^{n_1},\ldots,x_r^{n_r}$ for all $n_1,\ldots,n_r\in\Bbb N.$

\medskip

\noindent (ii) Note that every maximal $N$-sequence in dimension $>k$ in an ideal $I$ has the same length, this common length is denoted by $\depth_k(I,N)$ (see \cite{CHK}). Furthermore, $\depth_{-1}(I,N)$ is the usual depth $\depth(I,N)$ of $N$ in $I$, $\depth_0(I,N)$ is the filter depth $\f-depth(I,N)$ of $N$ in $I$ denoted by L${\rm\ddot u}$-Tang in \cite{LT}, and $\depth_1(I,N)$ is just the generalized depth $\gdepth(I,N)$ of $N$ in $I$ defined by Nhan \cite{Nh}.
\end{remark}

\begin{lemma}\label{L21} {\rm (\cite[Lemma 2.3]{CHK})} Let $k\ge -1$ be an integer. Then 
$$\depth_k(I,N)=\inf\{\depth_{k-i}(I_\p,N_\p)\mid\p\in\Supp_R(N/IN)_{\ge i}\}$$
for all $0\le i\le k+1$, where we use the convention that $\inf(\emptyset)=\infty$.
\end{lemma}

 The following two results are useful in the sequel.

\begin{lemma}\label{L31} {\rm (cf. \cite[Proposition 5.5]{BZ})} The following equality is true
$$\depth(I_M,N)=\inf\{i\mid H^i_I(M,N)\not=0\}.$$
\end{lemma}

\begin{lemma}\label{L32} {\rm (\cite[Theorem 2.4]{CH})} Let $r=\depth(I_M,N)$. Assume that $r<\infty$ and let $x_1,\ldots,x_r$ be a regular sequence of $N$ in $I_M$. Then
$$\Ass_R\big(H^r_I(M,N)\big)=\Ass_R\big(N/(x_1,\ldots,x_r)N\big)\cap V(I_M).$$
\end{lemma}

\section{Proof of Theorem \ref{T1} and its consequences}

The following theorem for generalized local cohomology modules implies Theorem \ref{T1} when $M=R$. 

\begin{theorem}\label{T11} Let $(R,\m)$ be a Noetherian local ring, $I$ an ideal of $R$ and $M, N$ finitely generated $R$-modules. Let $k$ be an integer with $k\ge -1$ and $r=\depth_k(I_M,N)$. If $r<\infty$ and $x_1,\ldots,x_r$ is an $N$-sequence in dimension $>k$ in $I_M$, then for any integer $l\le r$ we have 
$$\bigcup_{j\le l}\Ass_R(H^j_I(M,N))_{\ge k}=\bigcup_{j\le l}\Ass_R(N/(x_1,\ldots,x_j)N)_{\ge k}\cap V(I_M).$$
Therefore $\Ass_R(H^j_I(M,N))_{\ge k}$  is a finite set for all $j\le r$.
\end{theorem}

\begin{proof} Let $l$ be an integer with $0\le l\le r$. For each $\p\in\bigcup_{j\le l}\Ass_R(H^j_I(M,N))_{\ge k}$ there exists an integer $j_0\le l$ such that $\p\in\Ass_R(H^{j_0}_I(M,N))$ and $\p\notin\Ass_R(H^j_I(M,N))$ for all $j<j_0.$ Assume that  $\p\notin\Ass_R(N/(x_1,\ldots,x_j)N)$ for all $j<j_0$, then $\p R_\p\notin\Ass_{R_\p}( (N/(x_1,\ldots,x_j)N)_\p)$. Hence  
$$\Ass_{R_\p}((N/(x_1,\ldots,x_j)N)_{\p})=\Ass_{R_\p}((N/(x_1,\ldots,x_j)N)_{\p})_{\ge 1}$$
for all $j<j_0$. Since $x_1/1,\ldots,x_{j_0}/1$ is a filter regular sequence of $N_\p$ by Lemma \ref{L21} and the hypothesis of $x_1,\ldots,x_{j_0}$, it follows that $x_1/1,\ldots,x_{j_0}/1$ is a regular sequence of $N_\p$, and so $\depth((I_M)_\p,N_\p)\ge j_0.$ Since $\p\in\Ass_R(H^{j_0}_I(M,N))$, $H^{j_0}_I(M,N)_\p\not=0.$ Thus $\depth((I_M)_\p,N_\p)\le j_0$ by Lemma \ref{L31}, and hence we obtain that $\depth((I_M)_\p,N_\p)= j_0.$ This yields by Lemma \ref{L32} that
$$\Ass_{R_\p}(H^{j_0}_I(M,N)_{\p})=\Ass_{R_\p}((N/(x_1,\ldots,x_{j_0})N)_\p)\cap V((I_M)_\p).$$
Hence $\p\in \Ass_R(N/(x_1,\ldots,x_{j_0})N)\cap V(I_M).$
Thus 
$$\bigcup_{j\le l}\Ass_R(H^j_I(M,N))_{\ge k}\subseteq\bigcup_{j\le l}\Ass_R(N/(x_1,\ldots,x_j)N)_{\ge k}\cap V(I_M).$$
Conversely, for any $\p\in\bigcup_{j\le l}\Ass_R(N/(x_1,\ldots,x_j)N)_{\ge k}\cap V(I_M),$ 
 there is an integer $e\le l$ such that $\p\in\Ass_R(N/(x_1,\ldots,x_e)N)$ and $\p\notin\Ass_R(N/(x_1,\ldots,x_j)N)$ for all $j<e$. So $\p R_\p\notin\Ass_{R_\p}((N/(x_1,\ldots,x_j)N)_\p)$ for all $j<e.$ Hence 
$$\Ass_{R_\p}((N/(x_1,\ldots,x_j)N)_{\p})=\Ass_{R_\p}((N/(x_1,\ldots,x_j)N)_{\p})_{\ge 1}$$ 
for all $j<e$. Therefore $x_1/1,\ldots,x_e/1$ is a regular sequence of $N_\p$, and so that $\depth((I_M)_\p,N_\p)\ge e$. Keep in mind that $\p \in\Ass_R(N/(x_1,\ldots,x_e)N)$, so $\depth((I_M)_\p, (N/(x_1,\ldots,x_e)N)_\p)=0.$ Thus we get $\depth((I_M)_\p,N_\p)=e$. It follows by Lemma \ref{L32} that $\p\in\Ass_R(H^e_I(M,N)).$ This implies that
$$\bigcup_{j\le l}\Ass_R(H^j_I(M,N))_{\ge k}\supseteq\bigcup_{j\le l}\Ass_R(N/(x_1,\ldots,x_j)N)_{\ge k}\cap V(I_M),$$ 
and  the theorem follows. 
\end{proof}

Theorem \ref{T1} and Theorem \ref{T11} lead to many consequences.  It was showed in \cite[Proposition 2.6]{BN}  that if $x_1,\ldots,x_r$ is a permutable $N$-sequence in dimension $>k$ then
$\bigcup_{n_1,\ldots,n_r\in\Bbb N}\Ass_R(N/(x_1^{n_1},\ldots,x_r^{n_r})N)_{\ge k}$ is a finite set. By virtue of Theorem \ref{T11} we can prove this result without the hypothesis that $x_1,\ldots,x_r$ is a permutable $N$-sequence in dimension $>k$.

\begin{corollary}\label{Csua} Let $k$ be an integer with $k\ge -1$ and
$x_1,\ldots,x_r$ an $N$-sequence in dimension $>k$ (not necessary permutable). Then
$$\bigcup_{j\le r}\Ass_R(N/(x_1^{n_1},\ldots,x_j^{n_j})N)_{\ge k}=\bigcup_{j\le
r}\Ass_R(N/(x_1,\ldots,x_j)N)_{\ge k}$$
for all $n_1,\ldots,n_r\in\n$. In particular, $\bigcup_{n_1,\ldots,n_r\in\Bbb
N}\bigcup_{j\le r}\Ass_R(N/(x_1^{n_1},\ldots,x_j^{n_j})N)_{\ge k}$
is a finite set.
\end{corollary}
\begin{proof} Let $n_1,\ldots,n_r$ be positive integers, then
$x_1^{n_1},\ldots,x_r^{n_r}$ is an $N$-sequence in dimension $>k$ by Remark
\ref{RM1}. For each integer $i$ with $0\le i\le r$, we set $I_i=(x_1,\ldots,x_i)$.
Since
\begin{alignat}{2}
\Ass_R(N/(x_1^{n_1},\ldots,x_i^{n_i})N)_{\ge
k}=\Ass_R(N/(x_1^{n_1},\ldots,x_i^{n_i})N)_{\ge k}\cap V(I_i)\notag
\end{alignat}
for  any $r$-tuple of positive integers $n_1,\ldots,n_r$, it follows
from Theorem \ref{T1} that 
\begin{alignat}{2}
\bigcup_{i\le r}\Big(\bigcup_{j\le i}\Ass_R(H^j_{I_i}(N))_{\ge
k}\Big)&=\bigcup_{i\le r}\Big(\bigcup_{j\le
i}\Ass_R(N/(x_1^{n_1},\ldots,x_j^{n_j})N)_{\ge k}\cap V(I_i)\Big)\notag\\
&=\bigcup_{j\le r}\Ass_R(N/(x_1^{n_1},\ldots,x_j^{n_j})N)_{\ge k}.\notag
\end{alignat}
Then the corollary follows from the fact  that the set on the left hand of the equality is independent of the
choice of $n_1,\ldots,n_r$.
\end{proof}

Setting $k=1$ in Corollary \ref{Csua} we get the following corollary which covers Theorem 3.1 of \cite{Nh}. 

\begin{corollary}\label{Csua1} Let $x_1,\ldots,x_r$ be a generalized regular sequence of $N$. Then 
$$\bigcup_{j\le r}\Ass_R(N/(x_1^{n_1},\ldots,x_j^{n_j})N)\cup\{\m\}=\bigcup_{j\le r}\Ass_R(N/(x_1,\ldots,x_j)N)\cup\{\m\}$$
for all $n_1,\ldots,n_r\in\Bbb N$. In particular, $\bigcup_{n_1,\ldots,n_r\in\Bbb N}\bigcup_{j\le r}\Ass_R(N/(x_1^{n_1},\ldots,x_j^{n_j})N)$ is a finite set.
\end{corollary}

\begin{corollary}\label{BN1} Let $k$ be an integer with $k\ge -1$. Set $r=\depth_k(\ann(M),N)$. If $r<\infty$ and $x_1,\ldots,x_r$ is an $N$-sequence in dimension $>k$ in $\sqrt{\ann(M)}$, then for any integer $l\le r$ we have 
$$\bigcup_{j\le l}\Ass_R(\Ext^j_R(M,N))_{\ge k}=\bigcup_{j\le l}\Ass_R(N/(x_1,\ldots,x_j)N)_{\ge k}\cap V(\ann(M)).$$ 
\end{corollary}
\begin{proof} Set $I=\sqrt{\ann(M)}$. It is easy to see that $I_M=\sqrt{I_M}=I$. Hence $r=\depth_k(I_M,N)$ and $x_1,\ldots,x_r$ is an $N$-sequence in dimension $>k$ in $I_M$. Moreover, we have in this case that $ \Ext^j_R(M,N)\cong H^j_I(M,N)$ by Lemma \ref{Lmbs}. Therefore the conclusion follows from Theorem \ref{T11}.
\end{proof}

Let $j\ge 0$,  $t>0$ be integers. Let  $\underline a=(a_1,\ldots,a_s)$ be elements  of $R$ and  $\underline t=(t_1,\ldots,t_s)$
an $s$-tuple of positive integers. For an ideal $I$ we set 
$$
\begin{aligned}
&T^j(I^t,N)=\Ass_R(\Ext^j_R(R/I^t,N))\text{ and }\notag\\
&T^j(\underline a^{\underline t}, N)=\Ass_R(\Ext^j_R(R/(a_1^{t_1},\ldots,a_s^{t_s}),N)).\notag
\end{aligned}
$$
Then we get the following relation between these sets.  

\begin{corollary}\label{BN2} Let $k\ge -1$ be an integer, $I$ an ideal of $R$ and $r=\depth_k(I,N)$. If $r<\infty$ and $x_1,\ldots,x_r$ is an $N$-sequence in dimension $>k$ in $\sqrt I$. Then, for any system of generators $a_1,\ldots,a_s$ of $I$ and all integer $l\le r$, we have  
\begin{alignat}{2}
\bigcup_{j\le l}T^j(I^t,N)_{\ge k}=\bigcup_{j\le l}\Ass_R(N/(x_1,\ldots,x_j)N)_{\ge k}\cap V(I)=\bigcup_{j\le l}T^j(\underline a^{\underline t},N)_{\ge k}\notag
\end{alignat}
for all $t\in\Bbb N$ and all $\underline t\in\Bbb N^s$. In particular, for any integer $j\le r$ the sets $\bigcup_{t\in\Bbb N}T^j(I^t,N)_{\ge k}\text{ and }\bigcup_{\underline t\in\Bbb N^s}T^j(\underline a^{\underline t}, N)_{\ge k}$ are contained in the finite set $$\bigcup_{i\le j}\Ass_R(N/(x_1,\ldots,x_i)N)_{\ge k}\cap V(I).$$
\end{corollary}

\begin{proof} Since  $\sqrt{I^t}=\sqrt I=\sqrt{(a_1^{t_1},\ldots,a_s^{t_s})}$ for all positive integer $t, t_1,\ldots,t_s$, 
thanks to Corollary \ref{BN1} we obtain by setting $M=R/I^t$ and $M= R/(a_1^{t_1},\ldots,a_s^{t_s})$ that
$$\bigcup_{j\le l}\Ass_R(\Ext^j_R(R/I^t,N))_{\ge k}=\bigcup_{j\le l}\Ass_R(N/(x_1,\ldots,x_j)N)_{\ge k}\cap V(I)\ \text{ and }$$
$$\bigcup_{j\le l}\Ass_R(\Ext^j_R(R/(a_1^{t_1},\ldots,a_s^{t_s}),N))_{\ge k}=\bigcup_{j\le l}\Ass_R(N/(x_1,\ldots,x_j)N)_{\ge k}\cap V(I)$$
for all $l\le r$, and the corollary follows.
\end{proof}

\noindent
\begin{remark}\label{Rmrf}
The main result of Brodmann-Nhan \cite[Theorems 1.1 and 1.2]{BN} says  that  for a non negative integer $r$, if $\dim \Supp H_I^i(N)\le k$ for all $i <r$, then the sets $\bigcup_{t\in\Bbb N}T^j(I^t,N)_{\ge k}\text{ and }\bigcup_{\underline t\in\Bbb N^s}T^j(\underline a^{\underline t}, N)_{\ge k}$ are contained in the finite set $\bigcup _{i\le j}\Ass_R(\Ext^i_R(R/I,N))$ for  all $j\le r$. Moreover, if  $x_1,\ldots,x_r$ is at the same time  a permutable $N$-sequence in dimension $>k$ and a permutable $I$-filter regular sequence in  $I$, then these sets are contained in the finite set
\begin{equation}\Ass_R(N/(x_1,\ldots,x_j)N)_{\ge k+1} \cup\bigcup_{i\le j}\Ass_R(N/(x_1,\ldots,x_i)N)_k,\tag{*}\end{equation}
where $\Ass_R(N/(x_1,\ldots,x_i)N)_k=\{\p\in\Ass_R(N/(x_1,\ldots,x_i)N)\mid\dim R/\p=k\}$. In fact, since  the equality in Corollary \ref{BN2} do not depend on $t$ and $\underline t$, thus for any $j\le r$ the sets $\bigcup_{t\in\Bbb N}T^j(I^t,N)_{\ge k}$ and $\bigcup_{\underline t\in\Bbb N^s}T^j(\underline a^{\underline t}, N)_{\ge k}$ are contained in the set $\bigcup_{i\le j}T^i(I,N)_{\ge k}$ $=\bigcup _{i\le j}\Ass_R(\Ext^i_R(R/I,N))_{\ge k}$. This shows that Corollary \ref{BN2} covers Theorem 1.1 of \cite{BN}. 
Moreover, let $\p\in T^j(I^t,N)_{\ge k}\cup T^j(\underline a^{\underline t}, N)_{\ge k}$ for some $t$, $\underline t$. If $\dim(R/\p)=k$ then $\p$ belongs to the set $(*)$ by Corollary \ref{BN2}. If $\dim(R/\p)\ge k+1$ then $\depth(I_\p,N_\p)=r$. Therefore $j=r$ in this case, and so $\p\in T^r(I,N)$. Since $\Ext^r_R(R/I,N)\cong H^r_I(R/I,N)$ by Lemma \ref{Lmbs}, it follows by Lemma \ref{L32} that $\p\in\Ass(N/(x_1,\ldots,x_r)N)_{\ge k+1}$. Thus Corollary \ref{BN2} also covers Theorem 1.2 of \cite{BN} for local rings without the assumption that  $x_1,\ldots,x_r$ is at the same time  a permutable $N$-sequence in dimension $>k$ and a permutable $I$-filter regular sequence in  $I$. 
\end{remark}
\section{Proof of Theorem \ref{T2} and its consequences}

In this section we always denote by $\fR=\oplus_{n\ge 0}R_n$ a finitely generated standard graded algebra over $R_0=R$ and $\fN=\oplus_{n\ge 0}N_n$ a finitely generated graded $\fR$-module. Firstly, we recall two results on the asymptotic stability of associated prime ideals \cite{Br1} (see also \cite[Theorem 3.1]{Me}) and of generalized depths  $\depth_k(I,N_n)$ \cite[Theorem 1.1]{CHK} as follows. 
 
\begin{lemma} \label{L22} The set  $\Ass_R(N_n)$ is stable for large $n$.
\end{lemma}

\begin{lemma}\label{L23}{\rm (\cite[Theorem 1.1]{CHK})} Let $k$ be an integer with $k\ge -1$. Then $\depth_k(I,N_n)$ is independent of $n$ for large $n$.
\end{lemma}

\noindent In view of Lemma \ref{L23}, if $r=\depth_k(I,N_n)$ for all large $n$ then we call that $r$ is the eventual value of $\depth_k(I,N_n)$.

\begin{lemma}\label{L24} Let $k$ be an integer with $k\ge -1$ and $r$ the eventual value of $\depth_k(I,N_n)$. Assume that $1\le r\le\infty$. Then there exists a sequence $x_1,\ldots,x_r$ in $I$ which is an $N_n$-sequence in dimension $>k$ for all large $n$.
\end{lemma}

\begin{proof} Let $r\ge 1$ be the eventual value of $\depth_k(I,N_n)$.  By Lemma \ref{L22} we can choose $x_1\in I$ such that $x_1\notin\p$ for all $\p\in\Ass_R(N_n)_{>k}$  and all large $n$. Then $\depth_k(I,N_n/x_1N_n)=r-1$ for all large $n$. By the induction assumption, there exists a sequence $x_2,\ldots,x_n\in I$ to be an $N_n/x_1N_n$-sequence in dimension $>k$ for all large $n$. Thus $x_1,\ldots,x_r$ is the sequence as required.
\end{proof}

\noindent{\it Proof of Theorem \ref{T2}.} Theorem \ref{T2} follows immediately from the following theorem when $M=R$. 

\begin{theorem}\label{T22} Let $(R,\m)$ be a Noetherian local ring, $I$ an ideal of $R$ and $M$ finitely generated $R$-module. Let $\fR=\oplus_{n\ge 0}R_n$ be a finitely generated standard graded algebra over $R_0=R$ and $\fN=\oplus_{n\ge 0}N_n$ a finitely generated graded $\fR$-module. For each integer $k\ge -1$, let $r$ be the eventual value of $\depth_k(I_M,N_n)$. Then for each integer $l\le r$, the set $$\bigcup_{j\le l}\Ass_R(H^j_I(M,N_n))_{\ge k}$$ is finite and stable for large $n$.
\end{theorem}

\begin{proof} Let $r$ be the eventual value of $\depth_k(I_M,N_n)$.  We need to show that for any integer  $l\le r$ the set $\bigcup_{j\le l}\Ass_R(H^j_I(M,N_n))_{\ge k}$ is finite and stable for large $n$.

By Lemmas \ref{L22} and \ref{L23}, we can find a large integer $t$ such that $\Ass_R(N_n/I_MN_n)$ is stable and $r=\depth_k(I_M,N_n)$ for all $n\ge t$. Set $d=\dim(N_t/I_MN_t)$. We consider three cases as follows.

\medskip

\noindent{\it Case 1.} If $d<k$, then $r=\infty$ and the conclusion follows as $$\bigcup_{j\le l}\Ass_R(H^j_I(M,N_n))_{\ge k}=\emptyset$$ for all $n\ge t$.

\medskip

\noindent{\it Case 2.} If $d=k$, then $r=\infty$. It follows for all positive integers $n\ge t$ that 
$$\bigcup_{j\le l}\Ass_R(H^j_I(M,N_n))_{\ge d}\subseteq\Supp_R(N_n/I_MN_n)_{\ge d}=\Ass_R(N_n/I_MN_n)_{\ge d},$$
since the set $\Ass_R(N_n/I_MN_n)_{\ge d}$ only consists of all minimal elements with $\dim(R/\p)=d$. Hence 
$$X=\bigcup_{n\ge t}\bigcup_{j\le l}\Ass_R(H^j_I(M,N_n))_{\ge d}\subseteq\bigcup_{n\ge t}\Ass_R(N_n/I_MN_n),$$ 
therefore $X$ is finite by Lemma \ref{L22}. Thus we can take $t$ sufficiently large such that for each $\p\in X$ then 
$$\p\in\bigcup_{j\le l}\Ass_R(H^j_I(M,N_n))_{\ge d}$$ 
for infinitely many $n\ge t$. Now for each $\p\in X$ we set $s$ to be the eventual value of $\depth((I_M)_\p,(N_n)_\p)$, then we can write $s=\depth((I_M)_\p,(N_n)_\p)$ for all $n\ge n(\p)$ for some integer $n(\p)\ge t$. Then $H^s_I(M,N_n)_\p\not=0$ for all $n\ge n(\p)$. Moreover, by the definition of $X$, it implies that $s\le l$. Hence $\p$ is a minimal element of $\Supp_R(H^s_I(M,N_n))$, and so that 
$$\p\in\Ass_R(H^s_I(M,N_n))_{\ge d}\subseteq\bigcup_{j\le l}\Ass_R(H^j_I(M,N_n))_{\ge d}$$
for all $n\ge n(\p)$. Therefore the set 
$$\bigcup_{j\le l}\Ass_R(H^j_I(M,N_n))_{\ge d}$$ 
is finite and stable for all $n\ge\max\{n(\p)\mid\p\in X\}$.

\medskip

\noindent{\it Case 3.} Assume that $d>k$. Then $r<\infty$. If $l=0$, then 
$$\Ass_R(H^0_I(M,N_n))_{\ge k}=\Ass_R(N_n)_{\ge k}\cap V(I_M)$$ 
is finite and stable for large $n$ by Lemma \ref{L22}. Assume that $1\le l\le r$. In this case, by Lemma \ref{L24}, there exists a sequence $x_1,\ldots,x_r$ in $I_M$ which is an $N_n$-sequence in dimension $>k$ for all $n\ge u$ for some integer $u\ge t$. By Theorem \ref{T11}, we get the following equality
$$\bigcup_{j\le l}\Ass_R(H^j_I(M,N_n))_{\ge k}=\bigcup_{j\le l}\Ass_R(N_n/(x_1,\ldots,x_j)N_n)_{\ge k}\cap V(I_M)$$
for all $n\ge u$. Therefore we conclude by Lemma \ref{L22} that the following set
$$\bigcup_{j\le l}\Ass_R(H^j_I(M,N_n))_{\ge k}$$ 
is finite and stable for large $n$ as required.
\end{proof}

The rest of this section devotes to consider many consequences of Theorem \ref{T22}. By replacing $k=-1, 0, 1$ in Theorem \ref{T22}, we get the following result which is an extension of the main result in  \cite[Theorem 1.2]{CHK} for generalized local cohomology modules.

\begin{corollary}\label{C35} Let $r$, $r_0$ and $r_1$ be the eventual values of 
$\depth(I_M,N_n)$, $\f-depth(I_M,N_n)$ and $\gdepth(I_M,N_n)$, respectively. Then the following statements are true.
\begin{itemize}
\item[\rm (i)] The set $\Ass_R\big(H^r_I(M,N_n)\big)$ is finite and stable for large $n$.
\item[\rm (ii)] For each integer $l_0\le r_0$ the set $\bigcup_{j\le l_0}\Ass_R(H^j_I(M,N_n))$ is finite and stable for large $n$.
\item[\rm (iii)] For each integer $l_1\le r_1$ the set $\bigcup_{j\le l_1}\Ass_R(H^j_I(M,N_n))\cup\{\m\}$ is finite and stable for large $n$.
\end{itemize} 
\end{corollary}

\begin{proof} (i) Since $H^j_I(M,N_n)=0$ for all $j<r$ and all large $n$, so that
$$\bigcup_{j\le r}\Ass_R(H^j_I(M,N_n))_{\ge -1}=\Ass_R\big(H^r_I(M,N_n)\big)$$ 
is finite and stable for large $n$ by Theorem \ref{T22}.\\
(ii) The conclusion follows from Theorem \ref{T22} and the  equality
$$\bigcup_{j\le l_0}\Ass_R(H^j_I(M,N_n))_{\ge 0}=\bigcup_{j\le l_0}\Ass_R(H^j_I(M,N_n))$$
for all $l_0\le r_0$ and all large $n$.\\
(iii) For each $l_1\le r_1$ we have
$$\bigcup_{j\le l_1}\Ass_R(H^j_I(M,N_n))_{\ge 1}\cup\{\m\}=\bigcup_{j\le l_1}\Ass_R(H^j_I(M,N_n))\cup\{\m\}$$
for all large $n$. Therefore the result follows by Theorem \ref{T22}. 
\end{proof}

Note that if $I=\sqrt{\ann(M)}$ then $H^j_I(M,N_n)=\Ext^j_R(M,N_n)$ for all $j\ge 0$ by Lemma \ref{Lmbs}. Moreover, in this case we have $I_M=\sqrt{I_M}=I$. Therefore Theorem \ref{T22} leads to the following immediate consequence.

\begin{corollary}\label{Ce}  Let $k$ be an integer with $k\ge -1$ and $r$ the eventual value of $\depth_k(\ann(M),N_n)$. Then for any $l\le r$ the set $\bigcup_{j\le l}\Ass_R(\Ext^j_R(M,N_n))_{\ge k}$ is stable for large $n$.
\end{corollary}

For an integer $j\ge 0$, an ideal $I$ of $R$, a system $\underline a=(a_1,\ldots,a_s)$ of elements of $R$, and a $s$-tuple $\underline t=(t_1,\ldots,t_s)\in\Bbb N^s$, as in the previous section we set
$$
\begin{aligned}
&T^j(I^t,N_n)=\Ass_R(\Ext^j_R(R/I^t,N_n)),\notag\\
&T^j(\underline a^{\underline t}, N_n)=\Ass_R(\Ext^j_R(R/(a_1^{t_1},\ldots,a_s^{t_s}),N_n)).\notag
\end{aligned}
$$
Then we have
\begin{corollary}\label{BN3} Let $k\ge -1$ be an integer, $I$ an ideal of $R$ and $r$ the eventual value of $\depth_k(I,N_n)$. Let $a_1,\ldots,a_s$ be a system of generators of $I$ and an integer $l\le r$. Then for all positive integers $t, t_1,\ldots,t_s$ the following sets
$$\bigcup_{j\le l}T^j(I^t,N_n)_{\ge k}\text{ and }\bigcup_{j\le l}T^j(\underline a^{\underline t},N_n)_{\ge k}$$ 
are stable for large $n$. Moreover, if $r<\infty$ then 
$$\bigcup_{j\le l}T^j(I^t,N_n)_{\ge k}\text{ = }\bigcup_{j\le l}T^j(\underline a^{\underline t},N_n)_{\ge k}$$ 
are stable for all large $n$.
\end{corollary}

\begin{proof} By Lemma \ref{L23}, we find a large integer $u$ such that $r=\depth_k(I,N_n)$ for all $n\ge u$. Let $t, t_1,\ldots,t_s$ be positive integers. Since $\sqrt{I^t}=\sqrt I=\sqrt{(a_1^{t_1},\ldots,a_s^{t_s})}$\ , therefore 
$$r=\depth_k(\sqrt I,N_n)=\depth_k(I^t,N_n)=\depth_k((a_1^{t_1},\ldots,a_s^{t_s}),N_n)$$ 
for all $n\ge u$. This implies from Corollary \ref{Ce} by setting $M=R/I^t$ and $M= R/(a_1^{t_1},\ldots,a_s^{t_s})$ that $\bigcup_{j\le l}T^j(I^t,N_n)_{\ge k}$ and $\bigcup_{j\le l}T^j(\underline a^{\underline t},N_n)_{\ge k}$ are stable for large $n$. 
We now assume that $r<\infty$. Then, for any $n\ge u$, we have by Corollary \ref{BN2} that 
$\bigcup_{j\le l}T^j(I^t,N_n)_{\ge k}\text{ = }\bigcup_{j\le l}T^j(\underline a^{\underline t},N_n)_{\ge k}$ 
for all for all positive integers $t, t_1,\ldots,t_s$. Thus they
are stable for all large $n$.
\end{proof}

\noindent{\bf Acknowledgements:} The authors would like to thank the referee for pointing out a mistake in Remark \ref{Rmrf}.  This work is funded by Vietnam National Foundation for Science and Technology Development (NAFOSTED), and  the second author is also partially supported by Vietnam Ministry of Training and Education under grant number B2012-TN03-02.


\begin{thebibliography}{99}
{\small
\bibitem{BZ}  M. H. Bijan-Zadeh, {\it A common generalization of local cohomology theories,} Glasgow Math.  J., {\bf 21} (1980), 173-181.

\bibitem{Br1} M. Brodmann, {\it Asymptotic stability of $\Ass_R(M/I^nM),$} Proc. Amer. Math. Soc., (1) {\bf 74}
(1979), 16-18.

\bibitem{Br2} M. Brodmann, {\it The Asymptotic nature of the analytic spread}, Math. Proc. Camb. Phil. Soc., {\bf 86} (1979), 35-39.

\bibitem{BF} M. Brodmann and A. L. Faghani, {\it A Finiteness result for associated primes of local
cohomology modules}, Proc. Amer. Math. Soc., (10) {\bf 128} (2000), 2851 - 2853.

\bibitem{BN} M. Brodmann and L. T. Nhan, {\it A finiteness result for associated primes of certain $\Ext$-modules},  Comm. Algebra,  {\bf 36} (2008), 1527-1536.

\bibitem{CH}  N. T. Cuong and N. V. Hoang, {\it Some finite properties of generalized local cohomology modules,} East-West J. Math., (2) {\bf 7} (2005), 107-115.

\bibitem{CH1}  N. T. Cuong and N. V. Hoang, {\it On the vanishing and the finiteness of supports of generalized local cohomology modules,} Manuscripta Math., (1) {\bf 126} (2008),  59-72 .

\bibitem{CHK} N. T. Cuong, N. V. Hoang and P. H. Khanh, {\it Asymptotic stability of certain sets of associated prime ideals of local cohomology modules},  Comm. Algebra, {\bf 38} (2010), 4416-4429.

\bibitem{Cst} N. T. Cuong, P. Schenzel and N. V. Trung, {\it Verallgemeinerte Cohen-Macaulay Moduln,} Math. Nachr., {\bf 85} (1978), 57-73.

\bibitem{Her} J. Herzog,  {\it Komplexe, Aufl${\ddot o}$sungen und Dualit${\ddot a}$t in der Lokalen Algebra}, Habilitationsschrift, Universit${\rm\ddot a}$t Regensburg, 1970.

\bibitem{Hun} C. Huneke, {\it Problems on local cohomology}, Free resolutions in commutative algebra and algebraic geometry (Sundance, Utah, 1990), Res. Notes Math., {\bf 2} (1992), 93 - 108.

\bibitem{HuS} C. Huneke and R. Y. Sharp, {\it Bass numbers of local cohomology modules}, Trans. Amer. Math. Soc., {\bf 339} (1993), 765 - 779.

\bibitem{Kat} M. Katzman, {\it An example of an infinite set of associated primes of a local cohomology module,} J. Algebra, {\bf 252} (2002), 161-166.

\bibitem{KhS1}K. Khashyarmanesh and Sh. Salarian, {\it On the associated primes of local cohomology modules}, Comm. Algebra, {\bf 27} (1999), 6191 - 6198.

\bibitem{LT} R. L${\rm\ddot u}$ and Z. Tang, {\it The f-depth of an ideal on a module,} Math. Proc. Camb. Phil. Soc., (7) {\bf 130} (2001), 1905-1912.

\bibitem{Lyu} G. Luybeznik, {\it Finiteness properties of local cohomoly modules {\rm (}an application of D-modules to commutative algebra{\rm )}}, Invent. Math., {\bf 113} (1993), 41-55.

\bibitem{Mar} T. Marley, {\it Associated primes of local cohomology module over rings of small dimension}, Manuscripta Math., (4) {\bf 104} (2001), 519 - 525.

\bibitem{Me} L. Melkersson, {\it On asymptotic stability for sets of prime ideals connected with the powers of an ideal}, Math. Proc. Camb. Phil. Soc., {\bf 107} (1990), 267-271.

\bibitem{Nh} L. T. Nhan, {\it On generalized regular  sequences and the finiteness for associated primes of local cohomology
modules}, Comm. Algebra, {\bf 33} (2005), 793-806.

\bibitem{AS} A. Singh, {\it p-torsion elements in local cohomology modules}, Math. Res. Lett., {\bf 7} (2000), 165-176.
}
\end{thebibliography}
\end{document}